\documentclass[11pt]{amsart}
\usepackage{amssymb} 
\usepackage[mathscr]{eucal} 
\usepackage{xypic}   
\usepackage{amsmath} 
\usepackage[all, 2cell]{xy}
\usepackage{epsfig}
\usepackage{amscd}
\usepackage{mathrsfs}
\usepackage{tikz-cd}

\def\Q{{\mathbb Q}}
\def\Z{{\mathbb Z}}

\def\C{{\mathbb C}}

\def\F{{\mathbb F}}
\def\G{{\mathbb G}}

\newcommand{\Ql}{{{\overline{\mathbb Q}}_l}}

\newcommand{\Ee}{{\mathscr E}}
\newcommand{\Tt}{{\mathscr T}}

\newcommand{\Uu}{{\mathscr U}}

\newcommand{\Bb}{{\mathscr B}}
\newcommand{\Pp}{{\mathscr P}}

\newcommand{\Xx}{{\mathscr X}}

\newcommand{\Spec}{{\rm Spec}}

\newcommand{\Bun}{{\rm Bun}}

\newcommand{\BB}{{\Bb un_{G, X, \F_q}}}
\newcommand{\BBX}{{\Bb un^{\vartheta}_{G, X, \F_q}}}
\newcommand{\BBBX}{{\Bb un^{\vartheta}_{G, X, \overline{\F}_q}}}
\newcommand{\BBO}{{\Bb un}}
\newcommand{\BBBS}{{\Bb un}_{G, X, S}}

\newcommand{\BG}{{\Bb G}}

\newcommand{\fq}{{{\overline{\F}_q}}}
\newcommand{\qa}{{{\overline{\Q}_l}}}

\DeclareFontFamily{OT1}{pzc}{}
\DeclareFontShape{OT1}{pzc}{m}{it}{<->s*[1.21]pzcmi7t}{}
\DeclareMathAlphabet{\mathpzc}{OT1}{pzc}{m}{it}

\def\Gal{\mathrm{Gal}}

\def\Spec{\mathrm{Spec}}

\def\Frob{\mathrm{Frob}}

\def\Pic{\mathrm{Pic}}
\def\Jac{\mathrm{Jac}}

\def\GL{\mathrm{GL}}
\def\SL{\mathrm{SL}}

\def\dim{\mathrm{dim\,}}

\def\Pic{\mathrm{Pic}}
\def\cO{{\mathcal O}}

\def\E{{\mathcal E}}

\theoremstyle{definition}
\newtheorem{theorem}{Theorem}[section]

\newtheorem{proposition}[theorem]{Proposition}

\newtheorem{definition}[theorem]{Definition}

\newtheorem{example}[theorem]{Example}

\theoremstyle{remark}
\newtheorem{remark}[theorem]{Remark}

\numberwithin{equation}{section}

\begin{document}

\title[Actions of Frobenius morphisms for moduli of principal bundles]
{On actions of Frobenius morphisms for moduli stacks of principal bundles over algebraic curves}

\author[Abel Castorena]
{Abel Castorena}
\address{Centro de Ciencias Matem\'aticas
\\ Universidad Nacional Aut\'onoma de M\'exico (UNAM) Campus Morelia
\\Morelia, Michoac\'an 
C.P. 58089, M\'exico
}
\email{abel@matmor.unam.mx}

\author[Frank Neumann]
{Frank Neumann}
\address{Dipartimento di Matematica 'Felice Casorati'\\
Universit\`a di Pavia\\
via Ferrata 5, 27100 Pavia, Italy}
\email{frank.neumann@unipv.it}



\begin{abstract}
{We study the various arithmetic and geometric Frobenius morphisms on the moduli stack of principal bundles over a smooth projective algebraic curve and determine explicitly their actions on the $l$-adic cohomology of the moduli stack in terms of Chern classes.}
\end{abstract}
\maketitle


\section*{Introduction}

\noindent The moduli stack of principal $G$-bundles over an algebraic curve plays a fundamental role in algebraic geometry, number theory and mathematical physics. It appears eminently for example in relation with the geometric Langlands correspondence, the Weil conjecture on Tamagawa numbers for function fields or in gauge theory (see for example \cite{AB, Fr, BD, KW, DS, T, GaLu, HeSch}). In particular, the cohomology of the moduli stack is of fundamental interest and has been studied extensively. The fundamental work of Atiyah and Bott \cite{AB} implies in particular that for any semisimple complex reductive algebraic group $G$ the rational cohomology algebra of the moduli stack of principal $G$-bundles on a given Riemann surface $X$ is freely generated by the K\"unneth components of the Chern classes of the universal principal $G$-bundle over the moduli stack. This is basically a translation in terms of algebraic stacks and stack cohomology of the classical results obtained by Atiyah and Bott \cite{AB}, who originally used equivariant cohomology and Morse theory. A stacky exposition can be found in the article of Teleman \cite{T} and in the book of Gaitsgory and Lurie \cite{GaLu} on their proof of  Weil's Tamagawa number conjecture for function fields. Furthermore, Heinloth and Schmitt \cite{HeSch} gave an alternative proof also working in any characteristic. They showed in particular that the $l$-adic cohomology algebra of the moduli stack of principal $G$-bundles over an algebraic curve over a field in  characteristic $p$ is also freely generated as in the complex case by the K\"unneth components of Chern classes of the universal principal $G$-bundle over the moduli stack. 

In this article we study the various arithmetic and geometric Frobenius endomorphisms which act naturally on the moduli stack of principal $G$-bundles over an algebraic curve in finite characteristic and determine their explicit actions on the $l$-adic cohomology algebra of the moduli stack in terms of Chern classes. 

It turns out that besides the absolute arithmetic and geometric Frobenius morphism acting on the moduli stack, which in fact can be defined for general Artin algebraic stacks \cite{Be, Be2, Be3, Su}, there exist also an induced arithmetic and an induced geometric Frobenius morphism constructed by extending the arithmetic and geometric Frobenius morphism of the fixed algebraic curve to the whole moduli stack of principal $G$-bundles over the given curve.

All these four Frobenius  morphisms induce actions on the $l$-adic cohomology of the moduli stack which can explicitly be described in terms of cohomology classes. As one would expect, the absolute arithmetic and geometric Frobenius morphisms respectively the induced arithmetic and geometric Frobenius morphisms are inverse to each other as morphisms of algebraic stacks.

The geometry of these induced Frobenius morphisms on the moduli stack and its cohomology is rather mysterious. While for the absolute Frobenius morphisms, Behrend \cite{Be, Be2, Be3} derived a general Grothendieck-Lefschetz type trace formula, it is not known for example if there exists an adequate analogue of the trace formula which expresses the geometric and arithmetic properties of the induced Frobenius morphisms and especially would indicate which particular geometric structures are counted by these and what their arithmetic relevance is.  What can go wrong with a naive trace formula for the induced Frobenius morphisms already in basic examples for the case of vector bundles was discussed before by the second author and Stuhler \cite{NS} and a more systematic analysis is given in the last section of this article.\\

\noindent{\bf Outline and organization of the article.}In the first section we recall the main properties of the moduli stack of principal $G$-bundles over an algebraic curve in characteristic $p$ and its $l$-adic cohomology algebra. The second section features the constructions of the absolute and induced arithmetic and geometric Frobenius morphisms on the moduli stack. The third and final section gives the explicit calculations of the actions of the various Frobenius morphisms on the $l$-adic cohomology algebra and an analysis of the convergence of the formal trace for iterations of the induced arithmetic Frobenius and various compositions of the induced and absolute Frobenius morphism.

\section{The Moduli Stack of Principal $G$-bundles Over an Algebraic Curve and Its Cohomology}

In this section we recollect some fundamental facts about the moduli stack of principal $G$-bundles over an algebraic curve and the calculation of its $l$-adic cohomology algebra.

\subsection{The Moduli stack of principal G-bundles over an algebraic curve}

Let us first recall some important geometric properties of the moduli stack of principal $G$-bundles over a fixed algebraic curve. We will follow here the accounts in \cite{Be, Be2} and \cite{HeSch}.

\begin{definition}
Let $X$ be a smooth projective algebraic curve of genus $g$ over a noetherian scheme $S$ and $G/S$ be a reductive group scheme of finite rank over $S$. The {\it moduli stack} $\BBBS$ {\it of principal $G$-bundles over} $X$ is defined via its groupoid of sections as follows:
For a given scheme $U/S$ over $S$, the groupoid $\BBBS(U)$ of $U$-valued points of $\BBBS$ is the groupoid of principal $G$-bundles over $X\times_S U$ and their isomorphisms.
\end{definition}

We have the following fundamental fact (see \cite[Prop. 4.4.6, Cor.~4.5.2]{Be} or \cite[Prop. 3.4]{LaSo}):

\begin{proposition}
The moduli stack  $\BBBS$ is an Artin algebraic stack, locally of finite type and smooth of relative dimension
$(g-1)\dim(G)$ over $S$.
\end{proposition}

In general, the stack $\BBBS$ might not be connected, but its connected components can be described as follows (see \cite[Prop.~5]{DS}, \cite[Prop. 2.1.1]{HeSch}):

\begin{proposition}
If $S=\Spec(\overline{k})$, where $\overline{k}$ is an algebraic closure of a field $k$ or if $G$ is a split reductive algebraic group, then the set $\pi_0 (\BBBS)$ of connected components of  $\BBBS$ is in bijective correspondence with the set of elements of the fundamental group $\pi_1(G)$ of $G$.
\end{proposition}

As $\BBBS$ is smooth, it follows also that its connected components are irreducible algebraic stacks. 

Let $k$ be a field and for every $\vartheta\in \pi_1(G_{\overline{k}})$ let $\BBO^{\vartheta}_{G, X, \overline{k}}$ be the corresponding connected component of $\BBO_{G, X, \overline{k}}$. We have a decomposition:
$$\BBO_{G, X, \overline{k}}\cong \coprod_{\vartheta\in\pi_1(G_{\overline{k}})} \BBO^{\vartheta}_{G, X, \overline{k}}\;\;.$$  

We will later need to compare the moduli stack of principal $G$-bundles over a fixed algebraic curve with the classifying stack of all principal $G$-bundles, so we recall here briefly the definitions of quotient and classifying stacks. We refer to \cite{LMB} for the most general constructions and their basic geometric properties.

\begin{definition} Let $Z$ be a smooth scheme over a noetherian scheme $S$ and $G/S$ be a reductive group scheme of finite rank over $S$ together with a (right) $G$-action $\mu: Z\times_S G\rightarrow Z$. The {\it quotient stack} $[Z/G]$ is defined via its groupoid of sections as follows: For a given scheme $U/S$ over $S$, the groupoid $[Z/G](U)$ of $U$-valued points of $[Z/G]$ is the groupoid of principal $G$-bundles $P$ over $U$ together with a $G$-equivariant morphism $\alpha: P\rightarrow Z$ and isomorphisms of this data. For $Z=S$ equiped with the trivial $G$ action, we call $\BG:=[S/G]$ the {\it classifying stack} of $G$. 
\end{definition}

It turns out that the quotient stack $[Z/G]$ and in particular the classifying stack $\BG$ under the above conditions are again smooth algebraic stacks, which are locally of finite type (see \cite{LMB}).
 
\subsection{Cohomology of the moduli stack of principal $G$-bundles over an algebraic curve}

Let $\Xx$ be an algebraic stack, which is smooth and locally of finite type over $S=\Spec(k)$ for a field $k$. The {\it $l$-adic cohomology} defined over the lisse-\'etale site $\Xx_{\text{lis-\'et}}$ of $\Xx$ is given as the limit of the $l$-adic cohomologies of all the open substacks $\Uu$ of finite type of the algebraic stack  $\Xx_{\overline{k}}=\Xx\times_{\Spec(k)}\Spec({\overline{k}})$ (see \cite{HeSch}), i.e. we have
$$H^*( \Xx_{\overline{k}}, \Ql)=\lim_{\substack{{\Uu\subset\Xx_{\overline{k}}},\\ 
\text{open, finite type}}}\hspace*{-0.5cm}H^*(\Uu, \Ql).$$
As a general reference for cohomology of algebraic stacks we refer to the book of Laumon and Moret-Bailly \cite{LMB} and especially for $l$-adic cohomology and its main properties to the general formalism of cohomology functors as developed by Behrend \cite{Be2}, \cite{Be3} and in subsequent work by Laszlo and Olsson \cite{LO1}, \cite{LO2}.

From now on we will assume that $S=\Spec(\F_q)$ for the finite field $\F_q$ with $q=p^s$ elements for a given prime $p$ and that $G$ is a semisimple algebraic group over $\F_q$.  Let us recall the following theorem essentially due to Steinberg \cite{St} (see for example \cite[Prop. 2.2.5]{HeSch} for a proof):

\begin{theorem}\label{CohBG}
Let $G$ be a semisimple algebraic group of rank $r$ over the field $\F_q$. There exist integers $d_1, \ldots, d_r$ and roots of unity 
$\varepsilon_1, \ldots, \varepsilon_r$ such that:
\begin{itemize}
\item[(i)] The number of $\F_q$-rational points of $G$ is given as:
$$\#G(\F_q)= q^{\dim G}\prod_{i=1}^r (1-\varepsilon_i q^{-d_i})$$
\item[(ii)] Let $\BG$ be the classifying stack of $G$. Then there is an isomorphism of graded $\overline{\Q}_l$-algebras
$$H^*(\BG_{\bar{\F}_q}, \overline{\Q}_l)\cong {\overline\Q}_l[c_1, \ldots, c_r]$$
where the $c_i\in H^{2d_i}(\BG_{\overline{\F}_q}, \overline{\Q}_l)$ are generators in even degrees. 
\end{itemize}
\end{theorem}

Because $G$ is assumed here to be semisimple, it follows that $d_i>1$ for all $i=1, 2, \ldots, r.$

The universal principal $G$-bundle $\Pp^{univ}$ over the algebraic stack $X_\fq\times \BBBX$ defines a morphism of algebraic stacks
$$u: X_\fq\times \BBBX \rightarrow {\BG_\fq}.$$
The characteristic classes of 
$\Pp^{univ}$ are then defined as the pullbacks
$$c_i(\Pp^{univ}):= u^*(c_i), \,\,i=1, 2,\ldots, r,$$
where the $c_i$ are the standard generators of the cohomology algebra of the classifying stack $\BG_\fq$ as in the theorem of Steinberg.

We can now use the K\"unneth decomposition for $l$-adic cohomology of algebraic stacks to determine the K\"unneth components for the Chern classes
$c_i(\Pp^{univ})$ of the universal bundle $\Pp^{univ}$.

For this, let us first choose a basis for the $l$-adic cohomology of the algebraic curve $X$ given by the generators $1\in H^0(X_\fq, \qa)$, $(\gamma_j)_{j=1, 2, \ldots 2g}\in H^1(X_\fq, \qa)$ and the fundamental class $[X]\in H^2(X_\fq, \qa)$. We
have then the following K\"unneth decomposition for the Chern classes of the universal principal $G$-bundle $\Pp^{univ}$:
$$c_i(\Pp^{univ})=1\otimes a_i + \sum_{j=1}^{2g} \gamma_j \otimes b_i^{(j)} + [X]\otimes f_i,$$
where the classes  $a_i\in H^{2d_i}(\BBBX, \qa)$, 
$b_i^{(j)}\in H^{2d_i-1}(\BBBX, \qa)$ and 
$f_{i} \in H^{2(d_i-1)}(\BBBX, \qa)$ 
are the so-called {\em Atiyah-Bott classes}. 

These characteristic classes depend on the element $\vartheta\in \pi_1(G_\fq)$ and therefore on the connected components of the moduli stack $\BBO_{G, X, \F_q}$, but for simplicity we will drop $\vartheta$ from the notation of the classes.

The cohomology algebra of the moduli stack is then given as follows \cite{HeSch}:
\begin{theorem}
Let $G$ be a semisimple algebraic group of rank $r$ over the field $\F_q$ and $\vartheta\in \pi_1(G_\fq)$. Let $X$ be a smooth projective algebraic curve of genus $g$ over $\F_q$ and $\BBO_{G, X, \F_q}$ be the moduli stack of principal $G$-bundles on $X$. There is an
isomorphism of graded $\qa$-algebras
$$
\begin{aligned}
H^*(\BBBX, \qa)& \cong
\qa[a_1, \ldots, a_r]\otimes\qa[f_1, \ldots, f_r]\\
&\otimes\Lambda_{\qa}(b_1^{(1)}, \ldots, b_1^{(2g)}, \ldots,
b_r^{(1)}, \ldots, b_r^{(2g)}),
\end{aligned}
$$
\noindent where the generators $a_i, b_i^{(j)}$ and $f_i$ for $i=1, 2, \ldots, r$ are the Atiyah-Bott classes.
\end{theorem}

\begin{proof} This is the main theorem of \cite{HeSch}. It is a stacky version of the fundamental result of Atiyah and Bott \cite{AB} over the complex numbers (see also \cite{T}). For alternative proofs in the case of vector bundles, i.~e.~  for $G=\GL_n, \SL_n$ see also \cite{NS} or \cite{N}.
\end{proof}

\section{The Actions of the Arithmetic and Geometric Frobenius Morphisms on the Moduli Stack}

We will now define, relate and investigate the various arithmetic and geometric Frobenius morphisms of the moduli stack of principal $G$-bundles over a fixed algebraic curve. 

\subsection{The induced arithmetic Frobenius morphism}

Let us first have a look at the induced arithmetic Frobenius morphisms acting on the moduli stack $\BBX$ and its cohomology algebra, which is induced from the absolute Frobenius morphism on the algebraic curve $X$.

Let  $X$ be a scheme over $\Spec(\F_q)$ and let
$$\Frob: \overline{\F}_q \rightarrow \overline{\F}_q, \,\, a\mapsto a^q$$
be the classical Frobenius morphism given by a generator of the Galois group $\Gal(\overline{\F}_q/\F_q)$ of the field extension $\overline{\F}_q/\F_q$. 
It induces an endomorphism of affine schemes
$$\Frob_{\Spec(\overline{\F}_q)}: \Spec(\overline{\F}_q) \rightarrow \Spec(\overline{\F}_q)$$
which naturally induces an endomorphism of schemes via base change
$$\Frob_X:=id_{X} \times \Frob_{\Spec(\overline{\F}_q)}: X_{\overline{\F}_q} \rightarrow X_{\overline{\F}_q}.$$

Now more specially let $X$ be a smooth projective algebraic curve of genus $g$ over the field $\F_q$. By the above considerations we have the absolute arithmetic Frobenius endomorphisms associated to it:
$$\Frob_X: X_\fq\rightarrow X_\fq.$$

The pullback along this arithmetic Frobenius morphism $\Frob_X$ of the algebraic curve $X_\fq$ therefore induces an endofunctor of the groupoid of sections given by:
\begin{align*}
\psi(U): \BBBX(U)\rightarrow \BBBX(U)\\
\E\mapsto \psi(U)(\E):=(\Frob_X \times id_U)^* (\E)
\end{align*}

\noindent for each object $U$ of the category $(Sch/\fq)$ of schemes over the field $\fq$. 

These endofunctors assemble an endomorphism of algebraic stacks
$$\psi: \BBBX\rightarrow \BBBX.$$

We will call this endomorphism $\psi$ the {\em induced arithmetic Frobenius morphism} of the moduli stack. It induces an endormorphism in $l$-adic cohomology
$$\psi^*: H^*(\BBBX; \qa)\rightarrow H^*(\BBBX; \qa).$$

There is a relation between the arithmetic Frobenius morphism of the algebraic curve and the induced arithmetic Frobenius morphism of the moduli stack in form of a naturality property for the pullback of the universal principal bundle:

\begin{proposition}[Naturality]\label{Prop1}
There is a natural isomorphism
$$(\Frob_X\times id_{\BBBX})^*(\Pp^{univ})\cong (id_{X_\fq}\times \psi)^*
(\Pp^{univ})$$
where $\Pp^{univ}$ is the universal principal $G$-bundle on the stack $X_\fq\times \BBBX$.
\end{proposition}

\begin{proof} Representability implies that for any algebraic stack $\Tt$ over $\fq$ a principal $G$-bundle $\Pp$ on the stack $X_{\fq}\times \Tt$
is given by a morphism of stacks
$$u: \Tt \rightarrow \BBBX$$
such that $\Pp$ is isomorphic to the pullback of the universal principal $G$-bundle 
$\Pp^{univ}$ on the stack $X_\fq\times \BBBX$
$$\Pp\cong (id_{X_\fq} \times u)^* (\Pp^{univ}).$$
Applying this in particular to the pullback $(\Frob_X\times
id_{\BBBX})^*(\Pp^{univ})$ of the universal principal $G$-bundle $\Pp^{univ}$, which defines the endomorphism  $\psi: \BBBX\rightarrow \BBBX$, implies the statement.
\end{proof}

\subsection{The absolute arithmetic Frobenius morphism}

We have another arithmetic Frobenius morphism, which acts on a general algebraic stack $\Xx$ over $\Spec(\F_q)$ and its $l$-adic cohomology, the absolute Frobenius morphism of the algebraic stack. For this, let
$$\Frob: \overline{\F}_q \rightarrow \overline{\F}_q, \,\, a\mapsto a^q$$
denote the classical Frobenius morphism given by a generator of the Galois group $\Gal(\overline{\F}_q/\F_q)$ of the field extension $\overline{\F}_q/\F_q$. 
It induces an endomorphism of schemes
$$\Frob_{\Spec(\overline{\F}_q)}: \Spec(\overline{\F}_q) \rightarrow \Spec(\overline{\F}_q)$$
which naturally also extends to an endomorphism of algebraic stacks defined as
$$\Frob_{\Xx}:=id_{\Xx} \times \Frob_{\Spec(\overline{\F}_q)}: \Xx_{\overline{\F}_q} \rightarrow \Xx_{\overline{\F}_q}.$$
This endomorphism $\Frob_{\Xx}$ of algebraic stacks is called the {\em absolute arithmetic Frobenius morphism}.
It again induces an endomorphism on the $l$-adic cohomology of the algebraic stack
$$\Frob_{\Xx}^*: H^*(\Xx_{\overline{\F}_q}; \qa)\rightarrow H^*(\Xx_{\overline{\F}_q}; \qa).$$

An important example is the action of the absolute arithmetic Frobenius morphism $\Frob_{{\BG}_{\overline{\F}_q}}$ on the classifying stack $\BG_{\overline{\F}_q}$ and its induced action $\Frob_{{\BG}_{\overline{\F}_q}}^*$ on the $l$-adic cohomology $H^*(\BG_{\overline{\F}_q}, \overline{\Q}_l)$. In this case we have with the integers $d_1, \ldots, d_r$ and roots of unity $\varepsilon_1, \ldots \varepsilon_r$ as given in theorem \ref{CohBG} the following:

\begin{theorem}\label{AFrob}
The absolute arithmetic Frobenius $\Frob_{{\BG}_{\overline{\F}_q}}^*$ acts on the $l$-adic cohomology algebra of the classifying stack $\BG_{\overline{\F}_q}$
as follows:
$$
\begin{aligned}
\Frob_{{\BG}_{\overline{\F}_q}}^*(c_i)&= \varepsilon_iq^{-d_i}
c_i,
\end{aligned}
$$
\noindent where the $c_i$ for $i=1, 2, \ldots, r$ are the Chern classes.
\end{theorem}

\begin{proof} For a proof see for example \cite[Prop. 2.2.5]{HeSch}, \cite{Be2} or \cite{Su}.
\end{proof}

We are especially interested to determine the action of the arithmetic Frobenius endomorphism $\Frob_{\BBBX}$ on the moduli stack $\BBX$ and its induced action 
$$\Frob_{\BBBX}^*: H^*(\BBBX; \qa)\rightarrow H^*(\BBBX; \qa)$$
on the $l$-adic cohomology algebra.

\subsection{The induced geometric Frobenius morphism} 
The moduli stack of principal $G$-bundles also inherits in a natural way an induced geometric Frobenius morphism from the absolute geometric Frobenius acting on the algebraic curve itself.

Let us first recall the absolute geometric Frobenius morphism for a scheme $X$ over $\Spec(\F_q)$ (see for example \cite{M}).  It is an endomorphism of schemes defined as the identity on geometric points and as the $q$-th power operation on the structure sheaf, i.e. it is given as:
$$F_X: (X, \cO_X)\rightarrow (X, \cO_X),\, F_X=(id_X, f\mapsto f^q).$$

We also get an absolute geometric Frobenius morphism on the scheme $X_\fq=X\times_{\Spec(\F_q)} \Spec(\overline{\F}_q)$ via base change:
 $$\overline{F}_X=F_X\times id_{\Spec(\overline{\F}_q)}: X_\fq\rightarrow X_\fq.$$

Therefore given a smooth projective algebraic curve $X$ of genus $g$ over the field $\F_q$, we have the absolute geometric Frobenius endomorphisms associated to it:
$$\overline{F}_X: X_\fq\rightarrow X_\fq.$$

The pullback along this geometric Frobenius morphism $\overline{F}_X$ of the algebraic curve $X_\fq$ then again induces an endofunctor of the groupoid of sections given by
\begin{align*}
\varphi(U): \BBBX(U)\rightarrow \BBBX(U)\\
\E\mapsto \varphi(U)(\E):=(\overline{F}_X \times id_U)^* (\E)
\end{align*}
\noindent for each object $U$ of the category $(Sch/\fq)$ of schemes over the field $\fq$. 

These endofunctors give rise again to an endomorphism of algebraic stacks
$$\varphi: \BBBX\rightarrow \BBBX.$$

We will call this endomorphism $\varphi$ the {\em induced geometric Frobenius morphism} of the moduli stack. It again induces an endormorphism in $l$-adic cohomology
$$\varphi^*: H^*(\BBBX; \qa)\rightarrow H^*(\BBBX; \qa).$$

The relation between the geometric Frobenius morphism of the algebraic curve and the induced geometric Frobenius morphism of the moduli stack is given again by a naturality property for the pullback of the universal principal bundle on the moduli stack:

\begin{proposition}[Naturality]\label{Prop}
There is a natural isomorphism 
$$(\overline{F}_X\times id_{\BBBX})^*(\Pp^{univ})\cong (id_{X_\fq}\times \varphi)^*
(\Pp^{univ})$$
where $\Pp^{univ}$ is the universal principal $G$-bundle on the stack $X_\fq\times \BBBX$.
\end{proposition}

\begin{proof} Representability implies that for any algebraic stack $\Tt$ over the field $\fq$ a principal $G$-bundle $\Pp$ on the stack $X_{\fq}\times \Tt$
is given by a morphism of stacks
$$u: \Tt \rightarrow \BBBX$$
such that $\Pp$ is isomorphic to the pullback of the universal principal $G$-bundle 
$\Pp^{univ}$ on the stack $X_\fq\times \BBBX$, i.e.
$$\Pp\cong (id_{X_\fq} \times u)^* (\Pp^{univ}).$$
Applying this in particular to the pullback $(\overline{F}_X\times
id_{\BBBX})^*(\Pp^{univ})$ of the universal principal $G$-bundle $\Pp^{univ}$, which defines the endomorphism  $\varphi: \BBBX\rightarrow \BBBX$ implies the statement.
\end{proof}

As we will need it later let us recall here also how the geometric Frobenius morphism of the algebraic curve acts explicitly on its $l$-adic cohomology.

\begin{theorem}[Weil, Deligne]\label{Weil}
Let $X$ be a smooth projective curve of genus $g$ over the field $\F_q$ and let $X_\fq=X\times_{\Spec(\F_q)}
\Spec(\overline{\F}_q)$ be the associated curve defined over the algebraic closure $\overline{\F}_q$. Then we have:
$$
\begin{aligned}
H^0(X; \overline{\Q}_l) &= \overline{\Q}_l\cdot 1,\\
H^1(X; \overline{\Q}_l) &= \bigoplus_{i=1}^{2g} \overline{\Q}_l
\cdot \alpha_i,\\
H^2(X; \overline{\Q}_l) &=
\overline{\Q}_l\cdot [X],\\
H^i(X;
\overline{\Q}_l)&=0, \,\, \mathit{if}  \,\, i\geq 3.
\end{aligned}
$$
Here $[X]$ is the fundamental class and the $\alpha_i$ are eigenclasses under the
action of the geometric Frobenius morphism
$$\overline{F}^*_X: H^*(X, \overline{\Q}_l)
\rightarrow  H^*(X, \overline{\Q}_l)$$ given as
$$
\begin{aligned}
\overline{F}^*_X(1) &=1,\\
\overline{F}^*_X([X]) &= q [X],\\
\overline{F}^*_X(\alpha_i) &= \lambda_i \alpha_i\,\,
(i=1,2,\ldots 2g),
\end{aligned}
$$
where each $\lambda_i\in\overline{\Q}_l$ is an algebraic integer with $|\lambda_i|=q^{1/2}$ for any embedding of $\lambda_i$ in $\C$.
\end{theorem}

\begin{proof} For a detailed proof see for example the books of \cite{M} or \cite{FK}.
\end{proof}

\subsection{The absolute geometric Frobenius morphism}
Finally, we introduce the absolute geometric Frobenius morphism, which exists again for a general algebraic stack over $\Spec(\F_q)$.

So let $\Xx$ be an algebraic stack over $\Spec(\F_q)$. We construct an endomorphism of stacks
$$F_{\Xx}: \Xx\rightarrow \Xx$$
as follows: For any object $U$ of the category $(Sch/\F_q)$ of schemes over $\F_q$ we can interpret objects $x$ of the groupoid of sections $\Xx(U)$ as morphisms of stacks $x: U\rightarrow \Xx$ and we get therefore a natural endomorphism of groupoids
$$F_{\Xx}(U): \Xx(U)\rightarrow \Xx(U),$$
which maps the morphism $x: U\rightarrow \Xx$ to the morphism $x\circ F_U$, where $F_U: U\rightarrow U$ is the absolute geometric Frobenius
endomorphism of the scheme $U$. 

Base change extension then induces an endomorphism of algebraic stacks
$$\overline{F}_{\Xx}= F_{\Xx}\times id_{\Spec(\overline{\F}_q)}: \Xx_{\overline{\F}_q} \rightarrow \Xx_{\overline{\F}_q}.$$
We call this endomorphism the {\it absolute geometric Frobenius morphism} of the algebraic stack. By naturality it again induces an endomorphism on the $l$-adic cohomology of the algebraic stack $\Xx_{\overline{\F}_q}$:
$$\overline{F}_{\Xx}^*:  H^*(\Xx_{\overline{\F}_q}; \qa)\rightarrow H^*(\Xx_{\overline{\F}_q}; \qa).$$
The absolute geometric Frobenius $\overline{F}_{\Xx}^*$ is an automorphism and in fact the inverse of the absolute arithmetic Frobenius, i.e. $\overline{F}_{\Xx}^*=(\Frob_{\Xx}^*)^{-1}$ (see \cite{Be2, Be3}).

An important example is given by the action of the absolute geometric Frobenius morphism $\overline{F}_{{\BG}_{\overline{\F}_q}}$ on the classifying stack $\BG_{\overline{\F}_q}$ of a semisimple algebraic group $G$ over the field $k=\F_q$ of rank $r$ and its induced action $\overline{F}_{{\BG}_{\overline{\F}_q}}^*$ on the $l$-adic cohomology $H^*(\BG_{\overline{\F}_q}, \overline{\Q}_l)$. We have the following description of this action, where $d_1, \ldots, d_r$ are integers and $\varepsilon_1, \ldots \varepsilon_r$ roots of unity as given by theorem \ref{CohBG}:

\begin{theorem}
The absolute geometric Frobenius $\overline{F}_{{\BG}_{\overline{\F}_q}}^*$ acts on the $l$-adic cohomology algebra of the classifying stack $\BG_{\overline{\F}_q}$
as follows:
$$
\begin{aligned}
\overline{F}_{{\BG}_{\overline{\F}_q}}^*(c_i)&= \varepsilon_i^{-1}q^{d_i}
c_i,
\end{aligned}
$$
\noindent where the $c_i$ for $i=1, 2, \ldots, r$ are the Chern classes.
\end{theorem}

\begin{proof} As the absolute geometric Frobenius  $\overline{F}_{{\BG}_{\overline{\F}_q}}^*=(\Frob_{{\BG}_{\overline{\F}_q}}^*)^{-1}$ is the inverse of the absolute arithmetic Frobenius $\Frob_{{\BG}_{\overline{\F}_q}}^*$ the result follows readily from theorem \ref{AFrob}. 
\end{proof}


\section{The Actions of the Geometric and Arithmetic Frobenius Morphisms On the Cohomology of the Moduli Stack}

In this final section we will determine directly and explicitly the actions of the various geometric and arithmetic Frobenius morphisms on the $l$-adic cohomology algebra of the moduli stack and describe its effect on the Chern classes as well as their interactions.

\subsection{The action of the induced geometric Frobenius on the cohomology}
For the induced geometric Frobenius morphism $\varphi$ the action on the $l$-adic cohomology is described by the following theorem: 

\begin{theorem}[Induced Geometric Frobenius]
The geometric Frobenius $\varphi^*$ acts on the $l$-adic cohomology algebra of the moduli stack $\BBBX$
as follows:
$$
\begin{aligned}
\varphi^* (a_i)&=a_i,\\
\varphi^* (b_i^{(j)})&=\lambda_j b_i^{(j)}\,\,\, (j=1,\ldots, 2g),\\
\varphi^*(f_i)&= q f_i,
\end{aligned}
$$
\noindent where $a_i, b_i^{(j)}$ and $f_i$ for $i=1, 2, \ldots, r$ are the Atiyah-Bott classes.
\end{theorem}

\begin{proof}
Using functoriality of Chern classes and proposition \ref{Prop} we get the following equality
$$(\overline{F}_X\times
id_{\BBBX})^*(c_i(\Pp^{univ}))=
(id_{X_\fq}\times \varphi)^*(c_i(\Pp^{univ}))$$
Using K\"unneth decomposition for the Chern classes $c_i(\Pp^{univ})$ of the universal principal $G$-bundle $\Pp^{univ}$ we have:
$$c_i(\Pp^{univ})= 1\otimes a_i + \sum_{j=1}^{2g} \gamma_j\otimes
b_i^{(j)} + [X]\otimes f_{i}.$$ 
where the classes  
$a_i\in H^{2d_i}(\BBBX; \qa)$, 
$b_i^{(j)}\in H^{2d_i-1}(\BBBX; \qa)$ and 
$f_{i} \in H^{2(d_i-1)}(\BBBX; \qa)$ 
are the Atiyah-Bott classes. 

Evaluating the two expressions from the equality above and using theorem \ref{Weil} we get the following two equations
$$
(\overline{F}_X\times id_{\BBBX})^*(c_i(\Pp^{univ})) =
1\otimes a_i + \sum_{j=1}^{2g}\lambda_j \gamma_j\otimes
b_i^{(j)} + q[X]\otimes f_{i}
$$
$$
(id_{X_\fq}\times \varphi)^*(c_i(\Pp^{univ}))=1\otimes \varphi^*(a_i) + \sum_{j=1}^{2g} \gamma_j\otimes
\varphi^*(b_i^{(j)}) + [X]\otimes \varphi^*(f_{i}).
$$
Comparing now coefficients of the right hand sides of these two equations finally gives the desired description of the action of the Frobenius 
morphism $\varphi^*$ on the $l$-adic cohomology algebra of the moduli stack $\BBBX$.
\end{proof}

\subsection{The action of the absolute geometric Frobenius on the cohomology}

For the absolute geometric Frobenius morphism  the action on the $l$-adic cohomology of the moduli stack is described by the following theorem:

\begin{theorem}[Absolute Geometric Frobenius]
The absolute geometric Frobenius $\overline{F}_{\BB}^*$ acts on the $l$-adic cohomology algebra of the moduli stack $\BBBX$
as follows:
$$
\begin{aligned}
\overline{F}_{\BB}^*(a_i)&= \varepsilon_i^{-1}q^{d_i} a_i,\\
\overline{F}_{\BB}^*(b_i^{(j)})&=
\lambda_j^{-1} \varepsilon_i^{-1}q^{d_i} b_i^{(j)}\,\,\, (j=1,\ldots, 2g),\\
\overline{F}_{\BB}^*(f_i)&= \varepsilon_i^{-1}q^{d_i-1} f_i,
\end{aligned}
$$
\noindent where $a_i, b_i^{(j)}$ and $f_i$ for $i=1, 2, \ldots, r$ are the Atiyah-Bott classes.
\end{theorem}

\begin{proof}
Let  $\Ee_G^{univ}$ be the universal principal $G$-bundle over the classifying stack $\BG_{\fq}$ of the algebraic group $G$. 

The universal principal $G$-bundle $\Pp^{univ}$ over the algebraic stack $X_\fq\times \BBBX$
is then given via representability by a classifying morphism of algebraic stacks
$$u: X_\fq \times \BBBX\rightarrow \BG_{\fq}$$ 
with
$u^*(\Ee_G^{univ})\cong \Pp^{univ}$.
On the one hand we have a pullback diagram of $1$-morphisms of stacks related to the various actions of the geometric Frobenius morphisms, which is a $2$-cartesian diagram
$$\xymatrix{X_\fq\times \BBBX \ar[rr]^u \ar[d]_{\overline{F}_{X\times \BBX}}&&
\BG_\fq \ar[d]^{\overline{F}_{\BG}}\\
X\times \BBBX\ar[rr]^u   && \BG_\fq}
$$
inducing the following commutativity law via naturality for the respective geometric Frobenius morphisms: 
$$\overline{F}_{\BG}\circ u\cong u\circ \overline{F}_{X \times \BBX}.$$
Now evaluating the geometric Frobenius morphism $\overline{F}_{X\times \BBX}$ on the pullback of the universal principal $G$-bundle and using naturality of Chern classes gives then:
$$
\begin{aligned}
(\overline{F}_{X\times \BBX})^*(c_i(u^*(\Ee_G^{univ}))) &=
(\overline{F}_{X\times \BBX})^*(u^*(c_i(\Ee_G^{univ})))\\
&= u^*(\overline{F}_{\BG}^*(c_i(\Ee_G^{univ})))\\
&= \varepsilon_i^{-1}q^{d_i} u^*(c_i(\Ee_G^{univ}))\\
&= \varepsilon_i^{-1}q^{d_i} c_i(u^*(\Ee_G^{univ}))\\
&= \varepsilon_i^{-1}q^{d_i} c_i(\Pp^{univ})\\
&= \varepsilon_i^{-1}q^{d_i} (1\otimes a_i + \sum_{j=1}^{2g} \gamma_j\otimes
b_i^{(j)} + [X]\otimes f_{i}).
\end{aligned}
$$
On the other hand we also have the following relation between the geometric Frobenius morphisms evaluated on the Chern classes:
$$
\begin{aligned}
\overline{F}_{X\times
\BBX}^*(c_i(u^*(\Ee_G^{univ})))\\
&\hspace*{-1.5cm}=\overline{F}_{X\times
\BBX}^*(c_i({\Pp}^{univ}))\\
&\hspace*{-1.5cm}=(\overline{F}_X\times
id_{\BBBX})^*(id_{X_\fq}\times
\overline{F}_{\BBX})^*(c_i(\Pp^{univ})).
\end{aligned}
$$
Therefore we get now the following expression:
$$
\begin{aligned}
(\overline{F}_X\times
id_{\BBBX})^*(id_{X_\fq}\times
\overline{F}_{\BBX})^*(c_i(\Pp^{univ}))&\\
&\hspace*{-7cm}=(id_{X_\fq}\times
\overline{F}_{\BBX})^*(\overline{F}_X\times
id_{\BBBX})^*(c_i(\Pp^{univ}))\\
&\hspace*{-7cm}=(id_{X_\fq}\times
\overline{F}_{\BBX})^*(\overline{F}_X\times
id_{\BBBX})^*(1\otimes a_i + &\\
& &\hspace*{-5cm}+\sum_{j=1}^{2g} \gamma_j\otimes
b_i^{(j)} + [X]\otimes f_{i})\\
&\hspace*{-7cm}=(id_{X_\fq}\times
\overline{F}_{\BBX})^*(1\otimes a_i +\sum_{j=1}^{2g}\lambda_j \gamma_j\otimes
b_i^{(j)} + q[X]\otimes f_{i}).
\end{aligned}
$$
Comparing coefficients gives again the description of the action of the geometric Frobenius morphism $\overline{F}_{\BB}^*$ as stated in the theorem.
\end{proof}

\subsection{The action of the induced arithmetic Frobenius on the cohomology}

We calculate now the action of the induced arithmetic Frobenius morphism $\psi^*$ on the $l$-adic cohomology of the moduli stack. We have the following theorem:

\begin{theorem}[Induced Arithmetic Frobenius]
The induced arithmetic Frobenius $\psi^*$ acts on the $l$-adic cohomology algebra of the moduli stack $\BBBX$
as follows:
$$
\begin{aligned}
\psi^*(a_i)&= a_i,\\
\psi^*(b_i^{(j)})&= \lambda_j^{-1} b_i^{(j)}\,\,\,
(j=1,\ldots, 2g),\\
\psi^*(f_i)&= q^{-1} f_i,
\end{aligned}
$$
\noindent where the $a_i^{(j)}, b_i$ and $c_i$ for $i=1, 2, \ldots, r$ are the Atiyah-Bott classes.
\end{theorem}

\begin{proof} Functoriality of Chern classes and proposition \ref{Prop1} give the following: 
$$(\Frob_X\times id_{\BBBX})^*(c_i(\Pp^{univ}))=(id_{X_\fq}\times \psi)^*
(c_i(\Pp^{univ}))$$

Using K\"unneth decomposition for the Chern classes $c_i(\Pp^{univ})$ of the universal principal $G$- bundle $\Pp^{univ}$ we have as before
$$c_i(\Pp^{univ})= 1\otimes a_i + \sum_{j=1}^{2g} \gamma_j\otimes
b_i^{(j)} + [X]\otimes f_{i},$$ 
where the classes  
$a_i\in H^{2d_i}(\BBBX; \qa)$, 
$b_i^{(j)}\in H^{2d_i-1}(\BBBX; \qa)$ and 
$f_{i} \in H^{2(d_i-1)}(\BBBX; \qa)$ 
are the Atiyah-Bott classes. 

Evaluating again the two expressions from the above equality of Chern classes and using also the fact that the arithmetic Frobenius on the algebraic curve $X$ acts as the inverse of the geometric Frobenius, i.e. $\Frob_X^*=(\overline{F}_X^*)^{-1}$ we get from theorem \ref{Weil} the following equality of Chern classes:
$$(\Frob_X\times id_{\BBBX})^*(c_i(\Pp^{univ}))=
1\otimes a_i + \sum_{j=1}^{2g}\lambda_j^{-1} \gamma_j\otimes
b_i^{(j)} + q^{-1}[X]\otimes f_{i}.
$$
But otherwise we also have:
$$
(id_{X_\fq}\times \psi)^*(c_i(\Pp^{univ}))=1\otimes \psi^*(a_i) + \sum_{j=1}^{2g} \gamma_j\otimes
\psi^*(b_i^{(j)}) + [X]\otimes \psi^*(f_{i}).
$$
Comparing coefficients of the right hand sides of these two equations finally gives the desired description of the action of the induced arithmetic Frobenius  morphism $\psi^*$.
\end{proof}

\subsection{The action of the absolute arithmetic Frobenius on the cohomology} 

Finally, we analyze the effect of the arithmetic Frobenius morphism  $\Frob_{\BBBX}^*$ on the $l$-adic cohomology algebra of the moduli stack.
The action is completely described by the following theorem:
\begin{theorem}[Absolute Arithmetic Frobenius]
The absolute arithmetic Frobenius $\Frob_{\BBBX}^*$ acts on the $l$-adic cohomology algebra of the moduli stack $\BBBX$
as follows:
$$
\begin{aligned}
\Frob_{\BBBX}^*(a_i)&= \varepsilon_i q^{-d_i} a_i,\\
\Frob_{\BBBX}^*(b_i^{(j)})&= \lambda_j \varepsilon_i q^{-d_i} b_i^{(j)}\,\,\,
(j=1,\ldots, 2g),\\
\Frob_{\BBBX}^*(f_i)&= \varepsilon_i q^{-d_i+1}
f_i,
\end{aligned}
$$
\noindent where the $a_i^{(j)}, b_i$ and $c_i$ for $i=1, 2, \ldots, r$ are the Atiyah-Bott classes.
\end{theorem}

\begin{proof}
The relations follow from the previous theorem by observing that the arithmetic Frobenius morphism $\Frob_{\BBBX}^*$ is the
inverse of the absolute geometric Frobenius morphism $\overline{F}_{\BBX}^*$ (see \cite{Be}, \cite{Be2} and \cite{Su}). 

Alternatively, we can also argue directly in the following way:
Let  $\Ee_G^{univ}$ be again the universal principal $G$-bundle over the classifying stack $\BG_{\fq}$ of the algebraic group $G$. 

The universal principal $G$-bundle $\Pp^{univ}$ over the algebraic stack $X_\fq\times \BBBX$
is given via representability by a classifying morphism of algebraic stacks
$$u: X_\fq \times \BBBX\rightarrow \BG_{\fq}$$ 
with
$u^*(\Ee_G^{univ})\cong \Pp^{univ}$.

On the one hand we have the $2$-cartesian diagram
$$\xymatrix{X_\fq\times \BBBX \ar[rr]^u \ar[d]_{\Frob_{X\times \BBX}}&&
\BG_\fq \ar[d]^{\Frob_{\BG}}\\
X\times \BBBX\ar[rr]^u && \BG_\fq}
$$
inducing the following commutativity law via naturality for the respective arithmetic Frobenius morphisms: 
$$\Frob_{\BG}\circ u\cong u\circ \Frob_{X \times \BBX}.$$
Now evaluating the arithmetic Frobenius morphism $\overline{F}_{X\times \BBX}$ on the pullback of the universal principal $G$-bundle and using naturality of Chern classes gives:
$$
\begin{aligned}
(\Frob_{X\times \BBX})^*(c_i(u^*(\Ee_G^{univ}))) &=
(\Frob_{X\times \BBX})^*(u^*(c_i(\Ee_G^{univ})))\\
&= u^*(\Frob_{\BG}^*(c_i(\Ee_G^{univ})))\\
&= \varepsilon_i q^{-d_i} u^*(c_i(\Ee_G^{univ}))\\
&= \varepsilon_i q^{-d_i} c_i(u^*(\Ee_G^{univ}))\\
&= \varepsilon_i q^{-d_i} c_i(\Pp^{univ})\\
&= \varepsilon_i q^{-d_i} (1\otimes a_i + \sum_{j=1}^{2g} \gamma_j\otimes
b_i^{(j)} + [X]\otimes f_{i}).
\end{aligned}
$$
But on the other hand we also have the following functorial relation between the absolute arithmetic Frobenius morphisms evaluated on the Chern classes
$$
\begin{aligned}
\Frob_{X\times
\BBX}^*(c_i(u^*(\Ee_G^{univ})))\\
&\hspace*{-2.6cm}=\Frob_{X\times
\BBX}^*(c_i({\Pp}^{univ}))\\
&\hspace*{-2.6cm}=(\Frob_X\times
id_{\BBBX})^*(id_{X_\fq}\times
\Frob_{\BBX})^*(c_i(\Pp^{univ})).
\end{aligned}
$$
Therefore we get now the following expression:
$$
\begin{aligned}
(\Frob_X\times
id_{\BBBX})^*(id_{X_\fq}\times
\Frob_{\BBX})^*(c_i(\Pp^{univ}))&\\
&\hspace*{-8cm}=(id_{X_\fq}\times
\Frob_{\BBX})^*(\Frob_X\times
id_{\BBBX})^*(c_i(\Pp^{univ}))\\
&\hspace*{-8cm}=(id_{X_\fq}\times
\Frob_{\BBX})^*(\Frob_X\times
id_{\BBBX})^*(1\otimes a_i + &\\
& \hspace*{-1.6cm}+\sum_{j=1}^{2g} \gamma_j\otimes
b_i^{(j)} + [X]\otimes f_{i})\\
&\hspace*{-8cm}=(id_{X_\fq}\times
\Frob_{\BBX})^*(1\otimes a_i +\sum_{j=1}^{2g}\lambda_j \gamma_j\otimes
b_i^{(j)} + q[X]\otimes f_{i})
\end{aligned}
$$
Finally, comparing coefficients gives the description of the action of the absolute Frobenius morphism $\Frob_{\BBBX}^*$ as desired.\end{proof}

To illustrate the actions of the various Frobenius morphisms let us look at the very special case of the moduli stack of line bundles on algebraic curves, i.e. the case of the algebraic group $G=\G_m$.

\begin{example}[Moduli stack of line bundles over an algebraic curve]
The moduli stack $\Bb un_{\G_m, X, \F_q}$ of principal $\G_m$-bundles, i.e. line bundles on $X$ decomposes as the disjoint union of the moduli stacks classifying line bundles of degree $d$ on $X$, i.e. we get
$$\Bb un_{\G_m, X, \F_q}\cong\coprod_{d\in \Z}\Bb un^{1, d}_ {\G_m, X, \F_q}.$$

A coarse moduli space for the algebraic stack $\Bb un_X^{1, d}$ is given by the Picard scheme 
$\Pic^d_X$ of $X$. This assembles into a $\G_m$-gerbe 
$$\Bb un_X^{1, d}\rightarrow \Pic^{d}_X.$$
Furthermore, there exists a Poincar\'e universal family on the product scheme $X\times \Pic_X^d$ which gives a section of the $\G_m$-gerbe and therefore a splitting isomorphism (see \cite{Ra}, \cite[Thm. G]{DrNa} and \cite{He}) of algebraic stacks
$$\Bb un_X^{1, d}\cong \Pic^{d}_X \times \Bb\G_m$$
We can now use the K\"unneth theorem to determine the cohomology of the moduli stack $\Bun_X^{1, d}$ and have
$$H^*(\Bb un_X^{1,d}, \qa)\cong H^*(\Pic^{d}_X, \qa)\otimes H^*(\Bb \G_m, \qa).$$
\noindent The cohomology of the Picard scheme $\Pic^d_X$ is isomorphic to the cohomology of the Jacobian $\Jac(X)$ of the algebraic curve $X$, but the Jacobian is an abelian variety and therefore
$$H^*(\Jac(X), \qa)\cong \Lambda_{\qa} (H^1(X, \qa)\cong \Lambda_{\qa}(b_1^{(1)}, \ldots, b_1^{(2g)}).$$
So we get for the cohomology of the moduli stack
$$H^*(\Bun_X^{1,d}, \qa)\cong \Lambda_{\qa}(b_1^{(1)}, \ldots, b_1^{(2g)})\otimes \qa[a_1].$$
From this we can calculate the actions of the absolute and induced Frobenius morphism as before. In the case $G=\G_m$ we also have that the degree $d_1=1$. For the induced arithmetic Frobenius $\psi$ we get:
$$
\begin{aligned}
\psi^*(a_1)&= a_1,\\
\psi^*(b_1^{(j)})&= \lambda_j^{-1} b_1^{(j)}\,\,\,
(j=1,\ldots, 2g).
\end{aligned}
$$
The action of the absolute arithmetic Frobenius $\Frob_{\Bb un_X^{1,d}}$ is given as:
$$
\begin{aligned}
\Frob_{\Bb un_X^{1,d}}^*(a_1)&= \varepsilon_1 q^{-1} a_i,\\
\Frob_{\Bb un_X^{1,d}}^*(b_1^{(j)})&= \lambda_j \varepsilon_1 q^{-1} b_1^{(j)}\,\,\,
(j=1,\ldots, 2g).
\end{aligned}
$$
\end{example}
\noindent The actions of the induced geometric respectively absolute geometric Frobenius morphism can be derived from the above as they are the inverses of the induced arithmetic respectively absolute arithmetic Frobenius morphism.

\subsection{Actions of Frobenius morphisms, compositions and formal traces} 

In this section we will be studying compositions of the various Frobenius morphisms on the $l$-adic cohomology algebra of the moduli stack of principal bundles. Recall that there is an isomorphism of graded $\qa$-algebras
$$
\begin{aligned}
H^*(\BBBX, \qa)& \cong
\qa[a_1, \ldots, a_r]\otimes\qa[f_1, \ldots, f_r]\\
&\otimes\Lambda^*_{\qa}(b_1^{(1)}, \ldots, b_1^{(2g)}, \ldots,
b_r^{(1)}, \ldots, b_r^{(2g)}),
\end{aligned}
$$
\noindent where the generators $a_i, b_i^{(j)}$ and $f_i$ for $i=1, 2, \ldots, r;\,\,  j=1, \ldots 2g$ are the Atiyah-Bott characteristic classes with $a_i\in H^{2d_i}(\BBBX, \qa), b_i^{(j)}\in H^{2d_i-1}(\BBBX, \qa), 
f_i\in H^{2(d_i-1)}(\BBBX, \qa)$. 

From theorem 3.3 we obtain directly for the $n$-fold iterated action $(\psi^*)^n\,\, (n\in\Bbb N)$ of the induced arithmetic Frobenius $\psi^*$ on the $l$-adic cohomology algebra
$$(\psi^*)^n: H^*(\BBBX; \qa)\rightarrow H^*(\BBBX; \qa)$$
the following explicit expressions
$$
\begin{aligned}
(\psi^*)^n(a_i)&=a_i,\\
(\psi^*)^n(b_i^{(j)})&=\lambda_j^{-n}\cdot b_i^{(j)} \,\,\, (j=1,...,2g),\\
(\psi^*)^n(f_i))&=q^{-n}\cdot f_i,
\end{aligned}
$$
\noindent We recall that every scalar $\lambda_j\in\qa$ which appears in the actions of the different Frobenius morphisms are algebraic integers with 
$|\lambda_j|=q^{\frac{1}{2}}$ for any embedding of $\lambda_j$ in $\Bbb C$.

\noindent Denote by $\text{tr}$ the trace of a linear map of algebras, then the formal trace  is given by
$$\text{tr}((\psi^*)^n; H^*(\BBBX; \qa))=\sum\limits_{i\geq 0}(-1)^i \text{tr}((\psi^*)^n; H^i(\BBBX; \qa))$$
\noindent The action of $(\psi^*)^n$ on $\qa[a_1, \ldots, a_r]$ and $\qa[f_1, \ldots, f_r]$ respectively is given as an $(r\times r)$-identity matrix $I_r$, and a matrix $q^{-n}\cdot I_r$ respectively. Furthermore, the exterior algebra has a decomposition as
$$\Lambda_{\qa}({b_i^{(j)}}: i=1,...,r, j=1,...,2g)=\bigoplus\limits_{m\geq 0}\Lambda^m_{\qa}({b_i^{(j)}}: i=1,...,r, j=1,...,2g)$$

\noindent  For $i=1,...,r, j=1,...,2g$ let $v_{2g(i-1)+j}:=b_i^{(j)}$, then for $1\leq i_i<i_2<\cdots<i_m\leq 2rg$ the elements $v_{i_1}\wedge\cdots\wedge v_{i_m}$ generate the exterior algebra in degree $m$, $\Lambda^m_{\qa}({b_i^{(j)}})$, where for every subscript $i_k$ there exists a subscript $l_{i_k}\in\{1,...,r\}$ and a supraindex $j_{i_k}\in\{1,...,2g\}$ such that $v_{i_k}=b^{j_{i_k}}_{l_{i_k}}$. 

\noindent The action 

$$(\psi^*)^n:\Lambda^m_{\qa}({b_i^{(j)}})\longrightarrow\Lambda^m_{\qa}({b_i^{(j)}})$$ 

\noindent on the elements $v_{i_1}\wedge\cdots\wedge v_{i_m}\in\Lambda^m_{\qa}({b_i^{(j)}})$ is given explicitly as

$$
\begin{aligned}
(\psi^*)^n(v_{i_1}\wedge\dots\wedge v_{i_m}) & = (\psi^*)^n((v_{i_1}))\wedge\dots\wedge(\psi^*)^n((v_{i_m}))\\
& = (\psi^*)^n(v_{i_1})\wedge\dots\wedge(\psi^*)^n(v_{i_m}))\\
&= (\psi^*)^n(b^{(j_{i_1})}_{l_{i_1}})\wedge\dots\wedge(\psi^*)^n(b^{(j_{i_m})}_{l_{i_m}})\\
&=\lambda_{j_{i_1}}^{-n}\cdots\lambda_{j_{i_m}}^{-n} b^{(j_{i_1})}_{l_{i_1}}\wedge\dots\wedge b^{(j_{i_m})}_{l_{i_m}}\\
&=\lambda_{j_{i_1}}^{-n}\cdots\lambda_{j_{i_m}}^{-n} v_{i_1}\wedge\dots\wedge v_{i_m}.
\end{aligned}
$$

\noindent Therefore the action of the iteration $(\psi^*)^n$ on the exterior algebra on elements in degree $m$ is
$$(\psi^*)^n(v_{i_1}\wedge\dots\wedge v_{i_m})=(\prod\limits_{t=1}^m\lambda_{j_{i_t}}^{-n})v_{i_1}\wedge\dots\wedge v_{i_m}
$$
\noindent The action $(\psi^*)^n$ on the exterior algebra  $\Lambda^m_{\qa}({b_i^{(j)}})$ is hence given by a diagonal matrix of size $\binom{2rg}{m}\times\binom{2rg}{m} $. From this action we get that for every subscript $1\leq i_1<i_2<\cdots<i_m\leq 2rg$  and every element $v_{i_1}\wedge\dots\wedge v_{i_m}\in \Lambda^m_{\qa}({b_i^{(j)}})$ we obtain a scalar $\lambda^{-n}_{j_{i_1}}\cdots\lambda_{j_{i_m}}^{-n}$, so that the action is given explicitly as the following diagonal matrix

\noindent $\text{diag}(\sum_{i_1}^{2rg}(\lambda_{j_{i_1}}^{-n}\cdot\lambda_{j_{i_2}}^{-n}\cdots\lambda_{j_{i_m}}^{-n}),\sum_{i_2}^{2rg}(\lambda^{-n}_{j_{i_1}}\lambda^{-n}_{j_{i_2}}\cdots\lambda_{j_{i_m}}^{-n}), ...,\sum_{i_m}^{2rg}(\lambda^{-n}_{j_{i_1}}\cdots\lambda_{j_{i_m}}^{-n})).$

\noindent The trace for this action can hence be described by

$$\sum\limits_{1\leq i_1<i_2<\cdots<i_m\leq 2rg}\lambda^{-n}_{j_{i_1}}\cdot\lambda^{-n}_{j_{i_2}}\cdots\lambda_{j_{i_m}}^{-n}$$
Hence the trace of the action of $\psi^n$ on the generators of the $l$-adic cohomology $H^*(\BBBX, \qa)$ is given explicitly by the expression
$$\text{tr}((\psi^*)^n; H^*(\BBBX; \qa)=r+rq^{-n}+\sum\limits_{1\leq i_1<i_2<\cdots<i_m\leq 2rg}\lambda^{-n}_{j_{i_1}}\cdot\lambda^{-n}_{j_{i_2}}\cdots\lambda_{j_{i_m}}^{-n}.$$

\noindent From this we get an explicit expression for the formal trace of the iteration $(\psi^*)^n$.
Namely, since $|\lambda_{j_{i_t}}|=q^{\frac{1}{2}}$ for every subscript 
$j_{i_t}$, we get

$$
\begin{aligned}
\big|\sum\limits_{1\leq i_1<i_2<\cdots<i_m\leq 2rg}(\lambda^{-n}_{j_{i_1}}\cdots\lambda^{-n}_{j_{i_m}})\big |& \leq\sum\limits_{1\leq i_1<i_2<\cdots<i_m\leq 2rg}1\cdot 
 q^{-\frac{mn}{2}}\\
& =q^{-\frac{mn}{2}}(\sum\limits_{i_1=1}^{i_2-1} 1)(\sum\limits_{i_2=i_1+1}^{i_3-1} 1)\cdots(\sum\limits_{i_m=i_{m-1}+1}^{2rg} 1)\\
& =(i_2-1)(i_3-i_2-1)\cdots(2rg-i_m)q^{-\frac{mn}{2}}\\
\end{aligned}
$$

\noindent As $1\leq i_1<i_2<\cdots< i_{m-1}<i_m\leq 2rg$, we can bound the scalar $(i_2-1)(i_3-i_2-1)\cdots(2rg-i_m)$ by a natural number independent of $m$, hence we see that the formal trace of the iteration $(\psi^*)^n$
$$\text{tr}((\psi^*)^n; H^*(\BBBX; \qa))= \sum_{m\geq 0}(-1)^i\text{tr}((\psi^*)^n; H^m(\BBBX; \qa)$$ 
converges absolutely for all $n>0$. Now the question arises if there is a Lefschetz type trace formula for the composition $(\psi^*)^n$ and what does the formal trace count. 

More generally, let us now look at compositions of various iterations of the induced arithmetic and the absolute arithmetic Frobenius morphisms.  Recall that from theorem 3.4 we know also the action of the absolute arithmetic Frobenius morphism 
$$\text{Frob}_{\BBBX}^*: H^*(\BBBX; \qa)\rightarrow H^*(\BBBX; \qa)$$
on the $l$-adic cohomology algebra of the moduli stack $\BBX$ which is given by:
$$
\begin{aligned}
\text{Frob}_{\BBBX}^*(a_i)&=\varepsilon_iq^{-d_i}a_i, \\
\text{Frob}_{\BBBX}^*(b_i^{(j)})&=\lambda_j\varepsilon_iq^{-d_i}\cdot b_i^{(j)}\hskip2mm (j=1,...,2g),\\
\text{Frob}_{\BBBX}^*(f_i)&=\varepsilon_iq^{-d_i+1}\cdot f_i=\varepsilon_iq^{-d_i}qf_i.
\end{aligned}
$$
Let us denote $\eta=\text{Frob}_{\BBBX}$. Thus when we consider $s$-times the action of $\eta^*$ we have

$$
\begin{aligned}
(\eta^*)^s(a_i)&=\varepsilon_i^s(q^{-d_i})^sa_i,\\
(\eta^*)^s(b_i^{(j)})&=\lambda_j^s\varepsilon_i^s(q^{-d_i})^s\cdot b_i^{(j)}\hskip2mm (j=1,...,2g),\\
(\eta^*)^s(f_i))&=\varepsilon_i^s(q^{-d_i})^sq^s\cdot f_i.
\end{aligned}
$$

For $s,n\in\Bbb N$, we will now study the composition $(\eta^*)^s\circ(\psi^*)^n$ of the iterations on the $l$-adic cohomology
$$(\eta^*)^s\circ(\psi^*)^n: H^*(\BBBX; \qa)\rightarrow H^*(\BBBX; \qa).$$
From the above individual computations we obtain for the action of the composition  $(\eta^*)^s\times(\psi^*)^n$ on the Atiyah-Bott classes explicitly 
$$
\begin{aligned}
(\eta^*)^s\circ(\psi^*)^n(a_i)&=\varepsilon_i^s(q^{-d_i})^sa_i,\\
(\eta^*)^s\circ(\psi^*)^n(b_i^{(j)})&=\lambda_j^{s-n}\varepsilon_i^s(q^{-d_i})^s\cdot b_i^{(j)}\hskip2mm (j=1,...,2g),\\
(\eta^*)^s\circ(\psi^*)^n(f_i)&=q^{s-n}(q^{-d_i})^s\varepsilon^s_if_i.
\end{aligned}
$$
So the formal trace is given by
$$\text{tr}((\eta^*)^s\circ(\psi^*)^n; H^*(\BBBX; \qa))=\sum\limits_{i\geq 0}(-1)^i \text{tr}((\eta^*)^s\times(\psi^*)^n; H^i(\BBBX; \qa)).$$
The action of $(\eta^*)^s\circ(\psi^*)^n$ on $\qa[a_1, \ldots, a_r]$ is represented by the diagonal matrix $I_r$ of size $r\times r$, 
$$I_r=\text{diag}_r=(\varepsilon^s_1(q^{-d_1})^s,...,\varepsilon^s_r(q^{-d_r})^s),$$ 

\noindent and the action on $\qa[f_1, \ldots, f_r]$ is given by the diagonal matrix $q^{s-n}\cdot I_r$. Furthermore, the exterior algebra has a decomposition as
\vspace*{0.8cm}
$$\Lambda_{\qa}^*({b_i^{(j)}}: i=1,...,r, j=1,...,2g)=\bigoplus\limits_{m\geq 0}\Lambda^m_{\qa}({b_i^{(j)}}: i=1,...,r, j=1,...,2g).$$

\noindent  For $i=1,...,r, j=1,...,2g$ denote $v_{2g(i-1)+j}:=b_i^{(j)}$, then for $1\leq i_1<i_2<\cdots<i_m\leq 2rg$, the elements $v_{i_1}\wedge\cdots\wedge v_{i_m}$ generate the exterior algebra in degree $m$, $\Lambda^m_{\qa}<{b_i^{(j)}}>$, where for every subscript $i_k$ there exists a subscript $l_{i_k}\in\{1,...,r\}$ and a supraindex $j_{i_k}\in\{1,...,2g\}$ 
such that $v_{i_k}=b^{j_{i_k}}_{l_{i_k}}$.

\noindent \vspace*{0.3cm}The action 

$$(\eta^*)^s:\Lambda^m_{\qa}({b_i^{(j)}})\longrightarrow\Lambda^m_{\qa}({b_i^{(j)}})$$ 
\vspace{0.3cm}

\noindent on the elements $v_{i_1}\wedge\cdots\wedge v_{i_m}\in\Lambda^m_{\qa}({b_i^{(j)}})$ is given by

\vspace*{0.5cm}
$$
\begin{aligned}
(\eta^*)^s(v_{i_1}\wedge\dots v_{i_m})&=(\eta^*)^s((v_{i_1}))\wedge\dots\wedge(\eta^*)^s((v_{i_m}))\\
&= (\eta^*)^s(v_{i_1})\wedge\dots\wedge(\eta^*)^s(v_{i_m}))\\
&= (\eta^*)^s(b^{(j_{i_1})}_{l_{i_1}})\wedge\dots\wedge(\eta^*)^s(b^{(j_{i_m})}_{i_{i_m}})\\
&=(\lambda_{j_{i_1}}^s\varepsilon_{l_{i_1}}^s(q^{-d_{l_{i_1}}})^s)\cdots(\lambda_{j_{i_m}}^s\varepsilon_{l_{i_m}}^s(q^{-d_{l_{i_m}}})^s)b^{(j_{i_1})}_{l_{i_1}}\wedge\dots\wedge b^{(j_{i_m})}_{l_{i_m}}\\
&=(\lambda_{j_{i_1}}\varepsilon_{l_{i_1}}q^{-d_{l_{i_1}}})^s\cdots(\lambda_{j_{i_m}}\varepsilon_{l_{i_m}}q^{-d_{l_{i_m}}})^s v_{i_1}\wedge\dots\wedge v_{i_m}.
\end{aligned}
\vspace*{1cm}$$

\noindent Since we have already derived the action of the induced arithmetic Frobenius morphism $\psi^*$ on the exterior algebra, 
the action of the composition $(\eta^*)^s\circ(\psi^*)^n$ on the exterior algebra of degree $m$ is therefore given as

$$(\eta^*)^s\circ(\psi^*)^n(v_{i_1}\wedge\dots\wedge v_{i_m})=(\prod\limits_{t=1}^m\lambda_{j_{i_t}}^{(s-n)}\varepsilon_{l_{i_t}}^s(q^{-d_{l_{i_t}}})^s) v_{i_1}\wedge\dots\wedge v_{i_m},$$

\noindent and this action is represented again by a diagonal matrix of size $\binom{2rg}{m}\times\binom{2rg}{m}$. 

\noindent Indeed, for every subscript $1\leq i_1<i_2<\cdots<i_m\leq 2rg$  and every element $v_{i_1}\wedge\dots\wedge v_{i_m}\in \Lambda^m_{\qa}({b_i^{(j)}})$ we obtain a scalar $\prod\limits_{t=1}^m\lambda_{j_{i_t}}^{(s-n)}\varepsilon_{l_{i_t}}^s(q^{-d_{l_{i_t}}})^s$ such that the action is given explicitly in diagonal form as
$$
\begin{aligned}
\text{diag}(\sum\limits_{i_1}^{2rg}(\prod\limits_{t=1}^m\lambda_{j_{i_t}}^{(s-n)}\varepsilon_{l_{i_t}}^s(q^{-d_{l_{i_t}}})^s),\sum\limits_{i_2}^{2rg}(\prod\limits_{t=1}^m\lambda_{j_{i_t}}^{(s-n)}&\varepsilon_{l_{i_t}}^s(q^{-d_{l_{i_t}}})^s),...\\
&...,\sum\limits_{i_m}^{2rg}(\prod\limits_{t=1}^m\lambda_{j_{i_t}}^{(s-n)}\varepsilon_{l_{i_t}}^s(q^{-d_{l_{i_t}}})^s).
\end{aligned}
$$

\noindent The trace for this action can hence be described explicitly by the following sum

$$\sum\limits_{1\leq i_1<i_2<\cdots<i_m\leq 2rg}(\prod\limits_{t=1}^m\lambda_{j_{i_t}}^{(s-n)}\varepsilon_{l_{i_t}}^s(q^{-d_{l_{i_t}}})^s).$$

So finally the trace of the action of the composition on the generators of the $l$-adic cohomology $H^*(\BBBX, \qa)$ is given explicitly by the expression

$$
\begin{aligned}
\text{tr}(\psi^n; H^*(\BBBX; \qa))=\sum\limits_{k=1}^r&\varepsilon^s_k(q^{-d_k})^s +q^{s-n}(\sum\limits_{k=1}^r\varepsilon^s_k(q^{-d_k})^s)+\\
&+\sum\limits_{1\leq i_1<i_2<\cdots<i_m\leq 2rg}(\prod\limits_{t=1}^m\lambda_{j_{i_t}}^{(s-n)}\varepsilon_{l_{i_t}}^s(q^{-d_{l_{i_t}}})^s).
\end{aligned}
$$

\noindent In theorem 1.5, $G$ is assumed to be a semisimple algebraic group over the field $\Bbb F_q$.  In this case there exist roots of unity $\varepsilon_1,...,\varepsilon_r$ and integers  $d_1,...,d_r$ such that $d_i>1$ for $i=1,...,r$,  therefore we have that $\sum_{t=1}^md_{l_{i_t}}>m$, that is, $-m>-\sum_{t=1}^md_{l_{i_t}}$, thus $$q^{-sm}>(q^{-\sum_{t=1}^md_{l_{i_t}}})^s.$$  Also we have that $|\lambda_{j_{i_t}}|=q^{\frac{1}{2}}$ for every subscript $j_{i_t}$, thus we have for the formal trace of the composition

$$
\begin{aligned}
\big|\text{tr}((\eta^*)^s&\circ(\psi^*)^n; H^*(\BBBX; \qa))\big |\leq\\&\leq \sum\limits_{k=1}^r (q^{-d_k})^s+q^{s-n}(\sum\limits_{k=1}^r (q^{-d_k})^s) +\sum\limits_{1\leq i_1<i_2<\cdots<i_m\leq 2rg}q^{\frac{m(s-n)}{2}}\cdot\prod\limits_{t=1}^m(q^{-d_{l_{i_t}}})^s\\
&= \sum\limits_{k=1}^r (q^{-d_k})^s+q^{s-n}(\sum\limits_{k=1}^r (q^{-d_k})^s) +q^{\frac{m(s-n)}{2}}\cdot[\sum\limits_{1\leq i_1<i_2<\cdots<i_m\leq 2rg}(q^{-\sum_{t=1}^md_{l_{i_t}}})^s]\\
&=\sum\limits_{k=1}^r (q^{-d_k})^s+q^{s-n}(\sum\limits_{k=1}^r (q^{-d_k})^s)+\\
&\hspace*{0.8cm}+q^{\frac{m(s-n)}{2}}(\sum\limits_{i_1=1}^{i_2-1}(q^{-\sum_{t=1}^md_{l_{i_t}}})^s)(\sum\limits_{i_2=i_1+1}^{i_3-1}(q^{-\sum_{t=1}^md_{l_{i_t}}})^s) \cdots(\sum\limits_{i_m=i_{m-1}+1}^{2rg}(q^{-\sum_{t=1}^md_{l_{i_t}}})^s)\\
&\leq\sum\limits_{k=1}^r (q^{-d_k})^s+q^{s-n}(\sum\limits_{k=1}^r (q^{-d_k})^s) +q^{\frac{m(s-n)}{2}}(i_2-1)(i_3-i_2-1)\cdots(2rg-i_m)q^{m(-ms)}\\
&=\sum\limits_{k=1}^r (q^{-d_k})^s+q^{s-n}(\sum\limits_{k=1}^r (q^{-d_k})^s) +q^{\frac{m(s-n)}{2}}q^{-m^2s}(i_2-1)(i_3-i_2-1)\cdots(2rg-i_m).
\end{aligned}
$$

\vspace*{0.3cm}
\noindent Since $1\leq i_1<i_2<\cdots< i_{m-1}<i_m\leq 2rg$, we can find a natural number independent of $m$ to bound the scalar $(i_2-1)(i_3-i_2-1)\cdots(2rg-i_m)$.
Hence we see from the convergence of geometric series that the formal trace of the composition $(\eta^*)^s\circ(\psi^*)^n$
$$\text{tr}((\eta^*)^s\circ(\psi^*)^n; H^*(\BBBX; \qa))= \sum_{m\geq 0}(-1)^i\text{tr}((\eta^*)^s\circ(\psi^*)^n; H^m(\BBBX; \qa))$$ 
\noindent converges absolutely for all $n>s$.

\begin{remark}
For the absolute arithmetic Frobenius a Grothendieck-Lefschetz trace formula was established by Behrend \cite{Be, Be2} and it calculates the mass or cardinality of the groupoid of rational points of the moduli stack $\BBO_{G, X}$. In fact, as Gaitsgory and Lurie  \cite{GaLu} have recently shown in their proof of Weil's Tamagawa conjecture for function fields that the conjecture reduces to a statement on the $l$-adic cohomology of the moduli stack $\BBO_{G, X}$ and the Grothendieck-Lefschetz trace formula calculates the mass of rational points of the moduli stack in terms of local data involving the Tamagawa number of the generic fiber of $G$. The formal traces for the induced arithmetic Frobenius or its composition with the absolute arithmetic Frobenius as outlined above though behave rather different and don't fit into a similar Grothendieck-Lefschetz trace formula. In fact, the derivation of the Grothendieck-Lefschetz trace formula for the absolute arithmetic Frobenius by Behrend uses particular Harder-Narasimhan type filtrations of open substacks of the moduli stack and the actions of the arithmetic Frobenius morphism $\eta$ on the cohomology algebras of these substacks, but the induced arithmetic Frobenius morphism $\varphi$ does not respect these open substacks and filtrations in general. Here the behaviour is rather mysterious and it would be interesting to analyse the formal trace when convergent and what arithmetic relevance it has. In the case of $G=SL_n$, i.e. for moduli stacks of vector bundles with trivial determinant this was addressed in \cite{NS}. The systematic general approach on the construction and geometry of trace formulas in this context as provided by Frenkel and Ng\^o in \cite{FrBCN} looks also relevant here.
\end{remark}

\section*{Acknowledgements}

\noindent The first author is supported by Grant PAPIIT UNAM IN100723  ``Curvas, sistemas lineales en superficies proyectivas y fibrados vectoriales" and CONACyT, M\'exico A1-S-9029 ``Moduli de curvas y curvatura en $A_g$". The second author would like to thank Centro de Ciencias Matem\'aticas (CCM), UNAM Campus Morelia, for the wonderful hospitality and financial support through the Programa de Estancias de Investigaci\'on (PREI) de la Direcci\'on General Asuntos del Personal Acad\'emico, DGAPA-UNAM.  He also likes to thank N. Nitsure (TIFR Mumbai) for his interest and support and U. Stuhler (G\"ottingen) for introducing him to this exciting circle of ideas.

\end{document}